\documentclass[12pt,a4paper]{amsart}
\usepackage{url,amssymb,latexsym,bm}
\usepackage{fullpage} 
\newtheoremstyle{plain} 
     {2ex}              
     {2ex}              
     {}         
     {}                 
     {\bfseries}        
     {}                 
     {1ex}              
     {\thmname{#1 }\thmnumber{#2}\thmnote{ \normalfont{(#3)}}.
}
\newtheoremstyle{remark}
     {2ex}              
     {2ex}              
     {}                 
     {}                 
     {\bfseries}        
     {}                 
     {1ex}              
     {\thmname{#1 }\thmnumber{#2}\thmnote{ \normalfont{(#3)}}.
}
\theoremstyle{plain}
\newtheorem{definition}{Definition}[section]

\newtheorem{proposition}[definition]{Proposition}

\newtheorem{theorem}[definition]{Theorem}

\newtheorem{corollary}[definition]{Corollary}

\newtheorem{example}[definition]{Example}

\numberwithin{equation}{section}
\newcommand{\nats}{\mathbb{N}}
\newcommand{\pnats}{\ensuremath{\nats_+}}
\newcommand{\ints}{\mathbb{Z}}

\newcommand{\comps}{\mathbb{C}}
\newcommand{\qproj}{\ensuremath{\mathbb{QP}}}

\newcommand{\cproj}{\ensuremath{\mathbb{CP}}}

\newcommand{\del}{\partial}

\newcommand{\mi}{\mathrm{i}}

\newcommand{\hl}[1]{\textnormal{\textbf{#1}}} 

\newcommand{\gtilde}{\ensuremath{\widetilde{g}}}

\newcommand{\htilde}{\ensuremath{\widetilde{h}}}

\newcommand{\sing}{\ensuremath{\mathcal{V}}}

\newcommand{\xhat}{\ensuremath{\widehat{x}}}

\begin{document}
\frenchspacing
\title{A new approach to asymptotics of Maclaurin coefficients of algebraic 
functions}
\author{Alexander Raichev}
\address{Department of Computer Science \\ University of Auckland \\
  Private Bag 92019 \\ Auckland \\ New Zealand}
\email{raichev@cs.auckland.ac.nz}
\author{Mark C. Wilson}
\address{Department of Computer Science \\ University of Auckland \\
  Private Bag 92019 \\ Auckland \\ New Zealand}
\email{mcw@cs.auckland.ac.nz}
\begin{abstract}
We propose a general method for deriving asymptotics of the Maclaurin series
coefficients of algebraic functions that is based on a procedure of K. V. 
Safonov 
and multivariate singularity analysis.
We test the feasibility of this this approach by experimenting on several 
examples.
\end{abstract}
\maketitle
\section{Introduction}\label{sec:intro}

Let $\pnats$ denote the set of (strictly) positive natural numbers. 
For $m\in\pnats$, $\beta\in\nats^m$, and $x\in\comps^m$, 
let $x_i$ denote component $i$ of $x$,
let
$\beta!  = \beta_1!\cdots\beta_m!$,
$n\beta = (n\beta_1,\ldots,n\beta_m)$,
$x^\beta = x_1^{\beta_1}\cdots x_m^{\beta_m}$,
and $\del^\beta = \del_1^{\beta_1} \cdots \del_m^{\beta_m}$,
where $\del_j$ is partial differentiation with respect to coordinate $j$.
For a domain $\Xi\subseteq\comps^m$ (an open connected set) containing the 
origin and a holomorphic function $f:\Xi\to\comps$, let 
$f_\beta = \del^\beta f(0)/\beta!$ for $\beta\in\nats^m$, 
the coefficients of the Maclaurin series for $f$.

In \cite{Safo2000} K. V. Safonov proved the following theorem.

\begin{theorem}\label{Safonov}
If $f:\subset\comps^{d-1}\to\comps$ is an algebraic function holomorphic in a 
neighborhood of the origin and every term of the Maclaurin series of $f(x)$ 
contains a factor of $x_1$, then there exists a rational function  
$F:\subset\comps^d\to\comps$ holomorphic in a neighborhood of the origin and a
matrix $M\in GL_{d-1}(\ints)$ with nonnegative entries such that the following 
relation between the Maclaurin series coefficients of $f$ and $F$ holds:
\[
f_\beta = F_{\gamma,\gamma_1},
\] 
where $\beta\in\nats^{d-1}$ and $\gamma = \beta M$.
\end{theorem}
\noindent
Loosely speaking, every $(d-1)$-variate algebraic power series is a generalized
diagonal of a $d$-variate rational power series.
Moreover, Safonov gave a constructive procedure for finding $F$ given an 
annihilating polynomial for $f$.

In this report we propose a new approach to deriving asymptotics of  
Maclaurin series coefficients of algebraic functions $f$, namely, (1) apply 
Safonov's procedure to $f$ to get a multivariate rational function $F$, then
(2) apply multivariate  singularity analysis to $F$ to derive the asymptotics of
the Maclaurin series coefficients of $F$ in the direction appropriate to $f$.
In this way we reduce deriving asymptotics of algebraic functions to deriving
asymptotics of rational functions.
This is promising, because current general methods for the former task exist 
only in the univariate case (see \cite[Chapter VII]{FlSe2009} for a summary of
these), whereas current general methods for the latter exist in the multivariate
case as well (see \cite{PeWi2008} for a summary of these).

We test the feasibility of this approach by experimenting on several specific 
algebraic functions.
In some cases we meet with difficulty because of gaps in the present theory
of multivariate singularity analysis for rational functions.
We note these gaps and indicate directions for further research.
\section{Examples}\label{sec:examples}

Let us apply our new approach to several specific algebraic functions, beginning
with univariate ones.
Their asymptotics are already known, so we can check our answers.
We used Maple 11 to do the calculations, and our worksheets are available at
\url{http://www.cs.auckland.ac.nz/~raichev/research.html}.

A description of Safonov's procedure, which employs resolution of singularities,
can be found in \cite{Safo2000}, and a description of multivariate singularity 
analysis for rational functions can be found in \cite{RaWi2012} and \cite{PeWi2008}.
We will assume you are familiar with this material. 
Warning: \cite{Safo2000} contains some confusing passages;
see Section~\ref{sec:comments} for a helpful annotation.

\begin{example}[Catalan numbers right-shifted]
Success.

Consider the algebraic function
\[ 
f(x_1) = \frac{1 -\sqrt{1-4x_1}}{2}.
\]
It is the generating function of the Catalan numbers multiplied by $x_1$, and so
its sequence of Maclaurin coefficients is the Catalan sequence prefixed by 0.
We seek asymptotics for $f_n$.

Since the generating function of the Catalan numbers is holomorphic in a
neighborhood of the origin, we may apply Safonov's procedure to $f$.
Doing so, we get that $f_n = F_{n,n}$, where 
\[
F(x_1,x_2) = \frac{x_2 (2x_2 -1)}{x_1 +x_2 -1}.
\]
So the Maclaurin coefficients of $f$ are the diagonal Maclaurin coefficients of
$F$.
Notice that this is always the case for univariate $f$ since $d=1$ implies
$M = (1)$ in Theorem~\ref{Safonov}.
Notice also that in this example $f$ is \hl{combinatorial}, that is, its 
Maclaurin coefficients are all nonnegative, but $F$ is not combinatorial.

To derive asymptotics for $f_n$, we derive asymptotics for $F_{n\alpha}$ with
$\alpha = (1,1)$.
To do this we apply multivariate singularity analysis to $F$.
Letting $H$ equal the denominator of $F$, we find that 
$\sing = \{x\in\comps^2 : H(x)=0\}$ is globally smooth. 
There is one critical point of $F$ for $\alpha$, namely $c=(1/2,1/2)$,
and $c$ is clearly strictly minimal.
Thus by the formulas of \cite{RaWi2012} (\cite{RaWi2012} formulas),
\[
F_{n\alpha} 
= 
4^n \left[ \frac{1}{4 \sqrt{\pi} n^{3/2}} +\frac{3}{32 \sqrt{\pi} n^{5/2}} 
+O\left(n^{-7/2}\right) \right],
\]
as $n\to\infty$, which is correct.
\end{example}

\begin{example}[Supertrees]
Failure.  Critical points not minimal.

Consider the polynomial
\[ 
P(y,x_1) = y^4 -2y^3 +(1 +2 x_1) y^2 -2yx_1 +4x_1^3. 
\]
It is an annihilating polynomial for the combinatorial generating function 
$f(x_1)$ for supertrees (see \cite{FlSe2009} for more details). 
The first few Maclaurin coefficients of $f$ are known to be 
$0, 0, 2, 2, 8, 18, 64, 188$.
Applying Safonov's procedure to $f$ via $P$ we get that $f_n = F_{n,n}$, where 
\[
F(x_1,x_2) 
= 
\frac{ 2x_1 x_2^2 (2x_2^5 x_1^2 -3x_2^3 x_1 +x_2 +2x_2^2 x_1 -1)}
{ x_2^5 x_1^2 -2x_2^3 x_1 +x_2 +2x_2^2 x_1 -2 +4x_1 }.
\]

To derive asymptotics for $f_n$, we derive asymptotics of $F_{n\alpha}$ where
$\alpha = (1,1)$.
Letting $H$ equal the denominator of $F$, we find that $\sing$ is globally 
smooth. 
There are three critical points of $F$ for $\alpha$, namely
$(3/16-(1/16)\sqrt{5}, 1+\sqrt{5})$, $(3/16+(1/16)\sqrt{5}, 1-\sqrt{5})$, and
$(1/8,2)$.
However, none of these points is minimal.
To see this, for each of these points $q$ we compute the first nonzero term of
the Maclaurin series of $\gtilde$ (which depends on $q$; see \cite{RaWi2012} for a 
definition of $\htilde$ and $\gtilde$) and get $a t^4$ with $a<0$.
Thus $\Re\gtilde < 0$ near 0, hence $|\htilde(t)| < |q_2|$ near 0, hence
there exists a point $q'\in\sing$ such that $|q'_1|=|q_1|$ and 
$|q'_2| < |q'_2|$, hence $q$ is not minimal.
So the \cite{RaWi2012} formulas are not guaranteed to work. 

Plugging $c = (1/8,2)$ into the formulas anyway we get
\[
F_{n\alpha} 
= 
4^n \left[ \frac{1-\mi}{8\Gamma(3/4) n^{5/4}} +O\left(n^{-7/4}\right) \right]
\]
which is nonreal and thus incorrect since $f_n \ge 0$.
The correct answer is
\[
F_{n\alpha} 
= 
4^n \left[ \frac{1}{8\Gamma(3/4) n^{5/4}} +O\left(n^{-7/4}\right) \right]
\]
(see \cite[Chapter VI]{FlSe2009} for instance).
\end{example}

\begin{example}[Basic univariate]\label{no_critical}
Failure.  No critical points.

Consider the simple combinatorial algebraic function
\[ 
f(x_1) = \frac{x_1}{\sqrt{1-x_1}}. 
\]
Applying Safonov's procedure to $f$ we get that $f_n = F_{n,n}$, where 
\[
F(x_1,x_2) 
= 
\frac{x_2 x_1 (-3x_2 +3x_2^2 x_1 -2 +2x_2 x_1 +x_1 -2x_2^2 +2x_2^3 x_1)}
{-x_2 +x_2^2 x_1 -2 +2x_2 x_1 +x_1}.
\]

To derive asymptotics for $f_n$ we derive asymptotics for $F_{n\alpha}$ where
$\alpha = (1,1)$.
Letting $H$ equal the denominator of $F$, we find that $\sing$ is globally 
smooth. 
However, there are no critical points of $F$ for $\alpha$.
Thus we can not apply the \cite{RaWi2012} formulas.

We try to get around this as follows.
When $\alpha = (1,\lambda)$ for $0<\lambda<1$, there are two critical points of
$F$, namely, 
$\Big(
\frac{1}{8}\lambda^2 +(\frac{1}{8}\lambda 
-\frac{3}{8})\sqrt{\lambda^2 -10\lambda +9} -\lambda +\frac{7}{8},
\frac{3 -3\lambda +\sqrt{\lambda^2 -10\lambda +9}}{2(\lambda-1)} 
\Big)$ 
and
$\Big(
\frac{1}{8}\lambda^2 -(\frac{1}{8}\lambda 
-\frac{3}{8})\sqrt{\lambda^2 -10\lambda +9} -\lambda +\frac{7}{8},
\frac{3 -3\lambda -\sqrt{\lambda^2 -10\lambda +9}}{2(\lambda-1)} 
\Big)$. 
Plugging the second and first critical points into the \cite{RaWi2012} formulas
and taking the limit as $\lambda\to 1^{-}$ we get 
\[
F_{n\alpha} 
= 
\frac{1}{\sqrt{2\pi n}} +\infty,
\]
and its negative, respectively.
So both expansions are clearly incorrect, but the leading term of the first is
curiously close to that of the correct expansion, which is
\[
F_{n\alpha} 
= 
\frac{1}{\sqrt{\pi n}} +O\left(n^{-3/2}\right)
\]
(see \cite[Chapter VI]{FlSe2009} for instance).

Alternatively, searching for critical points of $F$ for $\alpha=(1,1)$ in 
$\cproj^2$ yields two points, namely $(0,1,0)$ and $(0,0,1)$ in
homogeneous coordinates.
However, mapping these back down to $\comps^2$ eventually leads to a
trivariate generating function with cone point singularities,
a type of singularity for which there is presently no multivariate 
singularity analysis worked out. 
\end{example}

\begin{example}[Basic bivariate]
Success.

Consider the bivariate algebraic function
\[ 
f(x_1,x_2) = x_1 \sqrt{1-x_1-x_2},
\]
which appears in \cite[Example 2]{Safo2000}.
Following Safonov's procedure we get that $f_{n\beta} = F_{n\beta M}$ 
where $M = \big( \begin{smallmatrix} 1 & 0 \\ 1 & 1 \end{smallmatrix} \big)$
and 
\[
F(x_1,x_2,x_3)
=
\frac{x_3 x_1 (3x_3 +2 +x_1 +x_2 x_1 +2x_3^2)}{x_3 +2 +x_1 +x_2 x_1}.
\]

Letting $H$ equal the denominator of $F$, we find that 
$\sing$ is globally smooth. 
There is one critical point of $F$ for $\beta M$, namely
$(\frac{-\beta_1}{\beta_1 +\beta_2}, \frac{\beta_2}{\beta_1}, -1)$.
Since $F$ is not combinatorial, we have no shortcuts for checking whether or not
this point is minimal.
Proceeding anyway and plugging the critical point into the \cite{RaWi2012} formulas we get
\[
F_{n\beta M} 
= 
- \left[ 
\left( \frac{\beta_1 +\beta_2}{\beta_1} \right)^{\beta_1 +\beta_2} 
\left( \frac{\beta_1}{\beta_2} \right)^{\beta_2} 
\right]^n
\left[
\frac{\beta_1}{\sqrt{8}\pi (\beta_1+\beta_2)^2 \sqrt{\beta_1\beta_2} n^2}
+O\left(n^{-3}\right)
\right],
\]
as $n\to\infty$ uniformly as $\beta M$ normalized to length 1 varies within
a compact set in $\qproj^2$.
This expansion agrees with experimental checks.
\end{example}

\begin{example}[Another basic bivariate]
Failure.  No critical points.

Consider the bivariate algebraic function
\[
f(x_1,x_2) = \frac{x_1}{\sqrt{1-x_1-x_2}}.
\]
Following Safonov's procedure we get that $f_{n\beta} = F_{n\beta M}$ 
where $M = \big( \begin{smallmatrix} 1 & 0 \\ 1 & 1 \end{smallmatrix} \big)$
and 
{\tiny
\[
F(x_1,x_2,x_3)
=
\frac{x_3x_1(-3x_3+3x_3^2x_1+3x_3^2x_1x_2-2+2x_3x_1+2x_3x_1x_2+x_1
+x_1x_2-2x_3^2+2x_3^3x_1+2x_3^3x_1x_2)}
{-x_3+x_3^2x_1+x_3^2x_1x_2-2+2x_3x_1+2x_3x_1x_2+x_1+x_1x_2}.
\]
}

Letting $H$ equal the denominator of $F$, we find that 
$\sing$ is globally smooth. 
However there are no critical points of $F$ for $\beta M$.
Thus it is unclear how best to proceed in deriving asymptotics.
It seems we should first understand the no-critical-points phenomenon in the 
simpler bivariate setting of Example~\ref{no_critical}. 
\end{example}
\section{Comments and Questions}\label{sec:comments}

As we see, to implement our new approach for deriving asymptotics of 
Maclaurin coefficients of algebraic functions in full generality, several gaps
in the theory of multivariate singularity analysis of rational functions need to
be filled, namely,
\begin{itemize}
\item How does one check for minimality in the non-combinatorial case?
\item How does one handle non-minimal critical points?
\item How does one handle the case of no critical points?
\end{itemize}
Thus more research is required.

We end with an observation on Safonov's procedure and a short annotation of
\cite{Safo2000}.

\begin{proposition}
Let $F$ be the rational function from Theorem~\ref{Safonov}.
At a critical point of $F$ for 
$\alpha:=(m_1,\ldots,m_{d-1},m_1)\in\pnats^d$ whose last component is nonzero,
the numerator of $F$ vanishes.
\end{proposition}

\begin{proof}
From Safonov's procedure, $F(x)$ has the form ``$\text{polynomial in $x$} + 
(\text{power of $x$})\cdot R(x)$'', where 
\[ 
R(\xhat,y) 
= 
\frac{y^2}{k}
\frac{(\del_d P)(x_1y,x_2,\ldots,x_{d-1},y)}
{P(x_1y,x_2,\ldots,x_{d-1},y)}, 
\]
$y = x_d$ (for easy reading), $k\in\pnats$, $P$ is a 
polynomial, and the variable substitution is understood to occur after the 
partial differentiation.
By definition, the denominator of a rational function vanishes at a critical 
point.
Thus by the specific form of $F$ above, it suffices to prove the proposition for
the rational function $R(\xhat,y)$.

Let $J$ be the denominator of $R$.
Among the critical equations of $R$ for $\alpha$ are
\begin{align*}
J(\xhat,y)
&= 
0  \\
\alpha_1^{-1} x_1 \del_1 J(\xhat,y)     
&= 
\alpha_d^{-1} y \del_d J(\xhat,y),
\end{align*}
that is,
\begin{align*}
P(x_1y,x_2,\ldots,x_{d-1},y)
&= 
0  \\
x_1 y (\del_1 P)(x_1y,x_2,\ldots,x_{d-1},y)     
&= 
y \left[ x_1 (\del_1 P)(x_1y,x_2,\ldots,x_{d-1},y) 
+(\del_d P)(x_1y,x_2,\ldots,x_{d-1},y) \right].
\end{align*}
Thus at a critical point where $y\neq 0$, 
$P(x_1y,x_2,\ldots,x_{d-1},y) = 0$ and 
$(\del_d P)(x_1y,x_2,\ldots,x_{d-1},y) = 0$, as desired.
\end{proof}

\begin{corollary}\label{zero_term}
For a critical point of $F$ whose last component is nonzero, the first term in 
the \cite{RaWi2012} formulas for the asymptotic expansion of $F$ is zero.
\end{corollary}

\begin{proof}
By \cite[Theorem 3.4]{RaWi2012}, 
the first term in the \cite{RaWi2012} formulas for the asymptotic expansion of $F_{n\alpha}$ 
contains a factor of the numerator of $F$ evaluated at the critical point.
Thus by the previous proposition this term is zero.
\end{proof}

In retrospect we should have expected no critical points in 
Example~\ref{no_critical}, for in the presence of a critical point the 
\cite{RaWi2012} formulas and Corollary~\ref{zero_term} predict no $n^{-1/2}$ term in the expansion for $F_{n\alpha}$ whereas the correct expansion does have an $n^{-1/2}$ term.

\center{\textbf{Annotation to \cite{Safo2000}}}
\begin{itemize}
\item page 261, line 4 of the main body of the article:
  $z$ should be $z_1$
\item page 263, lines 2--5: it appears that Safonov is defining his term
  `A-diagonal' here
\item 264, 6: $\widetilde{R}$ should be $\widetilde{R}(w,z)$
\item 266, 1: `of the degree $q$' should be `of some degree $q$'
\item 266, 7: $\zeta^{m_n}$ should be $\zeta_n^{m_n}$ and the remark
  `for some $m_1,\ldots,m_n\in\nats$' should be added
\item 267, 12: the asterisks should be to the left of the rightmost parenthesis
\item 269, 9: $u(\zeta)$ should be $u(\upsilon,\zeta)$
\item 269, last line: here Safonov states that there exists such an $l$ but does
  not explain how to find it; it seems one need to look at a few terms of the
  Maclaurin expansions of both $h$ and $a$ and find the $l$ that works by trial
  and error; note that $l$ is needed for the definition of $R^{(2)}$ on page 270
\item 270, 2: $\zeta_1^{\nu_{n1}}$ should be $\zeta_1^{\nu_{1n}}$
\item 270, 11: $h_1(\zeta)$ should be $h^{(1)}(\zeta)$
\item 270, 13: the leftmost occurrence of $R^{(1)}(z_0,z)$ should be 
  $R^{(2)}(z_0,z)$
\item 270, 13: $z_n^{l_2}$ should be $z_n^{l_n}$
\item 271, 7: one of the $(pk)!$s should be a $(qk)!$
\end{itemize}
\bibliographystyle{amsalpha}    
\bibliography{combinatorics}
\end{document}